\documentclass[12pt]{amsart}
\usepackage{amsxtra,amssymb,mathrsfs,color,bbold}
\usepackage[numbers,sort&compress]{natbib}

\makeatletter

\newcommand{\Rmnum}[1]{\expandafter\@slowromancap\romannumeral #1@}
\makeatother
\usepackage{mathtools}

\theoremstyle{plain}
\newtheorem{theorem}{Theorem}[section]

\newtheorem{proposition}[theorem]{Proposition}
\newtheorem{corollary}[theorem]{Corollary}

\theoremstyle{definition}
\newtheorem{definition}[theorem]{Definition}

\usepackage[dvips]{graphics}
\usepackage{mathrsfs}
\usepackage{amssymb}
\usepackage{stmaryrd}
\usepackage{pifont}
\usepackage{amsmath}
\usepackage{hyperref}
\usepackage{amsthm}
\usepackage{cases}
\usepackage{mathtools}
\allowdisplaybreaks 
\title[AM-un]{AM-unbounded norm compact operators on Banach lattices}

\date{\today}


\author[Z. Wang]{Zhangjun Wang$^{1}$}
\address{$^1$ The first author:School of Mathematics, Southwest Jiaotong University,
	Chengdu, Sichuan,
	China, 610000.}
\email{zhangjunwang@my.swjtu.edu.cn}

\author[Z. Chen]{Zili Chen$^{2}$}
\address{$^2$ The second author:School of Mathematics, Southwest Jiaotong University, Chengdu, Sichuan,
	China, 610000.}
\email{zlchen@swjtu.edu.cn}

\author[J. Chen]{Jinxi Chen$^{3}$}
\address{$^3$ The third author: School of Mathematics, Southwest Jiaotong University, Chengdu, Sichuan,
	China, 610000.}
\email{jinxichen@swjtu.edu.cn}

\begin{document}

\begin{abstract}
In this paper, we introduce and study a new classes of operators, named \emph{AM-unbounded norm compact} operators. We study the the basic properties of the new operator and we investigate the lattice-order and topology property of the operator space of $AM$-unbounded norm compact operator.
\end{abstract}
	
\maketitle

\section{Introduction and Preliminaries}
	
Research of compact operators and weakly compact operators on Banach lattices has a long history in functional analysis literature. The notion of \emph{unbounded order convergence} was firstly introduced by Nakano in \cite{N:48}. Several recent papers investigated unbounded order , norm , and absolute weak convergence \cite{GX:14,G:14,KMT:16,Z:16}. The continuity and compactness of operators with respect to the unbounded convergence has been investigated in \cite{KMT:16,BA:18,OGZ:18}. Now, we consider the $AM$-compactness defined by unbounded convergence of operators.  

In this paper, using the unbounded norm convergence in Banach lattices, we introduce and study a new classes of operators, named \emph{AM-unbounded norm compact} operators. In section 2, we study the the relationship with orther operators, also we research the adjoint operator of the new operator. In section 3, we investigate the lattice and topology property of the operator space of $AM$-unbounded norm compact operator. When it is a band. We introduce a new norm named $AM$-$un$-norm, and found these operators form a Banach lattice under the new norm. Moreover, we gave the equivalent condition to describe when the operator space is $AM$ or $AL$-space.

Let us review the basic knowledge. 
A net $(x_\alpha)$ in $E$ is said to be \emph{unbounded order convergent} ($uo$-convergent, for short) to $x$ if for every $u\in E_+$ the net $(|x_\alpha-x|\wedge u)$ converges to zero in order. It is called \emph{unbounded norm convergent} ($un$-convergent, for short) to $x$ if $\Vert|x_\alpha-x|\wedge u\Vert\rightarrow0$ for every $u\in E_+$, and called \emph{unbounded absolute weak convergent} ($uaw$-convergent, for short) to $x$ if $\Vert|x_\alpha-x|\wedge u\Vert\xrightarrow{w}0$ for every $u\in E_+$.
	
$un$-convergence and $uaw$-convergence are topological. For every $\epsilon>0$ and non-zero $u\in E_+$, put 
$$V_{\epsilon,u}=\{x\in E:\Vert|x|\wedge u\Vert<\epsilon\}.$$
The collection of all sets of this form is a base of zero neighborhoods for a topology, and the convergence in this topology agrees with $un$-convergence, we will refer to this topology as $un$-topology.
We can also form $uaw$-topology by 
$$V_{\epsilon,u,f}=\{x\in E:f(|x|\wedge u)<\epsilon\}$$
where $u\in E_+$, $\epsilon>0$, $f\in E^{\prime}_+$, which form a base of zero neighborhoods for a Hausdorff topology, and the convergence in this topology is exactly the $uaw$-topology.
	
A bounded subset $A$ in Banach lattice $E$ is said $un(uaw)$-compact whenever every net $(x_\alpha)$ in $E$ has a subnet, which is $un(uaw)$-convergent, a bounded subset $A$ of Banach lattice $E$ is called to be sequentially $un(uaw)$-compact whenever every net $(x_n)$ in $E$ has a subsequence, which is $un(uaw)$-convergent.

Let $E$ be Riesz space and $F$ be Banach space, we say that $T:E\rightarrow F$ is $AM$-compact if $T[-x,x]$ is relatively compact subset of $F$ for any $x\in E_+$, and a operator $T:E\rightarrow F$ is order weakly compact if $T[-x,x]$ is relatively weakly compact subset of $F$ for any $x\in E_+$.
	
Let $E$ be Banach space and $F$ be Banach lattice, we say that $T:E\rightarrow F$ is (sequentially) $un$-compact if $TB_E$ is relatively (sequentially) $un$-compact in $F$ \cite{KMT:16}, and a operator $T:E\rightarrow F$ is (sequentially) $uaw$-compact if $TB_E$ is relatively (sequentially) $uaw$-compact in $F$ \cite{OGZ:18}.
	
For undefined terminology, notations and basic theory of Riesz spaces and Banach lattices, we refer to [1-4].	
\section{basic properties}\label{}
\begin{definition}\label{}
An continuous operator $T:E\rightarrow F$ from Riesz space $E$ to Banach lattice $F$ is said to be \emph{AM ($\sigma$)-unbounded norm compact} (AM-($\sigma$)-$un$-compact, for short), if  $T[-x,x]$ is (sequentially) $un$-compact in $F$ for any $x\in E_+$.
		
An continuous operator $T:E\rightarrow F$ from Riesz space $E$ to Banach lattice $F$ is said to be \emph{order ($\sigma$)-unbounded absolute weakly compact} ($o$-($\sigma$)-$uaw$-compact, for short), if  $T[-x,x]$ is (sequentially) $uaw$-compact in $F$ for any $x\in E_+$.
\end{definition}

It is clear that $AM$-$un$-compact operator is $o$-$uaw$-compact.
In fact, $uaw$-compact is different with sequentially $uaw$-compact, in a Hausdorff topological vector space which is metrizable, compact subset and sequentially compact subset are equivalent. Since the $un$-topology and $uaw$-topology are not metrizable in general, hence the classes operators is different on sequences and nets, but we still don't know what the essential difference is.

Recall that a positive element $e$ of Banach lattice $E$ is said to be \emph{quasi-interior point} if the ideal $E_e$ generated by $e$ is norm dense in $E$. $un$-topology on a Banach lattice $E$ is metrizable if and only if $E$ has a quasi-interior point \cite[Theorem~3.2]{KMT:16}. If $F$ has quasi-interior point, then a operator $T:E\rightarrow F$ from a Riesz space into $F$ is $AM$-$un$-compact iff $T$ is $AM$-$\sigma$-$un$-compact.  If $F$ is order continuous Banach lattice with quasi-interior point, then $o$-$uaw$-compact operator has a similar conclusion. Factly, most of the classical Banach lattices is satisfied, e.g. $l_p,l_\infty,c_0,L_p[0,1],L_\infty[0,1],C(K)$.
	
A vector $e>0$ in a Riesz space $E$ is said \emph{strong order unit} whenever the ideal generated by $e$ is $E$. It is easy to see that $un(uaw)$-compact operator $T:E\rightarrow F$ between two Banach lattices is $o$-$un(uaw)$-compact. The identical operator $I:c_0\rightarrow c_0$ is $AM$-$un$-compact and $o$-$uaw$-compact, but is is not $un$-compact and $uaw$-compact. If $E$ has strong order unit, then the converse hold because of the order interval and closed unit ball in $E$ are coincide.

$AM$-compact operator is $AM$-$un$-compact, but it is necessary. The identical operator $T:L_p[0,1]\rightarrow L_p[0,1]$ is $AM$-$un$-compact, but it is not $AM$-compact. If the value space has strong order unit, these are equivalent.	

By observing that there is no direct relationship between $o$-weakly compact operators and $o$-$uaw$-compact oeprators. For example, the continuous operator $T:l_\infty\rightarrow c_0$ as $T(x_n)=(\frac{1}{n}x_n)_1^{\infty}$ for every $(x_n)\in l_\infty$, $T$ is order weakly compact and is not $o$-$uaw$-compact. And considering the identical operator $I:ba(N)\rightarrow ba(N)$ ($ba(N)$ is the dual space of $l_\infty$), it is $o$-$uaw$-compact, but it is not order weakly compact. The next result shows that when a operator is both $o$-$uaw$-compact and order weakly compact.

Recall that norm on a Banach lattice $E$ is called \emph{order continuous} if $\Vert x_\alpha\Vert\rightarrow0$ for $x_\alpha\downarrow0$.	
\begin{proposition}\label{}
Let $T:E\rightarrow F$ be a continuous operators from a Riesz space $E$ into a Banach lattice $F$ and the lattice operations of $F$ is weakly sequentially continuous, if $F^{\prime}$ is order continuous then $T$ is both $o$-$\sigma$-$uaw$-compact operator and $o$-weakly compact operator, and in addition, if $T$ is onto, the converse holds.
\end{proposition}
\begin{proof}
$\Rightarrow$ For any $x\in E_+$, suppose that $T$ is $o$-weakly compact, then $T[-x,x]$ is relatively weakly compact subset of $F$, and so for any sequence $(y_n)$ of $T[-x,x]$ there exsits a subsequence $(z_n)$ such that $(z_n)$ is weak convergent in $F$, because of the lattice operations of $F$ is weakly sequentially continuous, hence $(z_n)$ is $uaw$-convergent in $F$, therefore $T$ is $o$-$\sigma$-$uaw$-compact. Assmue that $T$ is $o$-$\sigma$-$uaw$-compact, by \cite[Theorem~7]{Z:16}, $uaw$-convergence implies weak convergence on buonded sequence, we have the conclusion.

$\Leftarrow$ Assume that $F^{\prime}$ is not order continuous , by \cite[Theorem~2.4.14 and Proposition~2.3.11]{MN:91}, $l_1$ is a closed sublattice of $E$ and there exists a positive projection $P:F\rightarrow l_1$. Considering the standard basis $(e_n)$ of $l_1$, there exists $u\in E_+$ and $(x_n)\subset[-u,u]$, such that $Tx_n=e_n$ ($e_n\xrightarrow{uaw}0$ by \cite[Lemma~2]{Z:16}). Since $T$ is both $o$-$\sigma$-$uaw$-compact operator and $o$-weakly compact operator, so $P\circ T(x_{n_k})=e_{n_k}\xrightarrow{w}0$ in $l_1$ for some subsequence $(x_{n_k})$ of $(x_n)$ and $(e_{n_k})$ of $(e_n)$, it is contradiction, therefore $F^{\prime}$ is order continuous.
\end{proof}
Considering the operator $T:l_1\rightarrow l_\infty$ as $T(x_n)=(\sum_{n=1}^{\infty}x_n)_1^{\infty}$ for every $(x_n)\in l_1$, $T$ is $o$-$uaw$-compact , but it is not $AM$-$un$-compact. Then we have the following show that when $o$-$uaw$-compact operators is $AM$-$un$-compact.  
\begin{proposition}\label{un-uaw compact}
Let $E$ be Riesz space and $F$ be Banach lattices, if $F$ has order continuous then a $o$-($\sigma$)-$uaw$-compact operator is $AM$-($\sigma$)-$un$-compact, and in addition, if $T$ is onto, the converse holds.
\end{proposition}
\begin{proof}
$\Rightarrow$ By \cite[Theorem~4]{Z:16}, if Banach lattice $F$ has order continuous norm then the $un$-convergence is equivalent to $uaw$-convergence in $F$.

$\Leftarrow$ Assume that $F$ is not order continuous, then there exsits an order bounded disjoint sequence $(y_n)$ in $F$ is not convergent to zero. Since $o$-$uaw$-compact is interval-bounded, so there exists $(x_n)$ such that $(y_n=Tx_n)$ is contained in a $uaw$-compact subset of $F$, there exsits a subnet $(z_\gamma)$ is $un$-convergent. Because $(z_\gamma)$ is also order bounded disjoint net, therefore $z_\gamma\rightarrow0$, it is contradiction, so $F$ is order continuous.
\end{proof}
A Banach lattice $E$ is called \emph{KB}-space if every bounded increasing sequence in $E_+$ is convergent. According to \cite[Theorem~7.5]{KMT:16}, we have that every order bounded operator form a Banach lattice into a atomic $KB$-space is $AM$-$un$-compact and $AM$-$uaw$-compact operator. Similarly, order bounded $AM$-$un$-compact operator is $AM$-compact and order bounded $o$-($\sigma$)-$uaw$-compact operator maps order interval to relatively absolute weakly compact subset.

Recall that a continuous $T$ between two Banach lattices $E$ and $F$ is said \emph{($\sigma$)-strong continuous} operator whenever $Tx_\alpha\rightarrow0(Tx_n\rightarrow0)$ in $F$ for every bounded $un$-null net(sequence) $(x_\alpha)((x_n))$ in $E$ and \emph{($\sigma$)-unbounded norm continuous} operator whenever $Tx_\alpha\xrightarrow{un}0(Tx_n\xrightarrow{un}0)$ in $F$ for every bounded $un$-null net(sequence) $(x_\alpha)((x_n))$ in $E$ \cite{W:19}. A continuous $T$ between two Banach lattices $E$ and $F$ is said \emph{($\sigma$)-unbounded Dunford-Pettis} operator whenever $Tx_\alpha\xrightarrow{un}0(Tx_n\xrightarrow{un}0)$ in $F$ for every bounded $uaw$-null net(sequence) $(x_\alpha)((x_n))$ in $E$ and \emph{($\sigma$)-unbounded absolute weak continuous} operator whenever $Tx_\alpha\xrightarrow{uaw}0(Tx_n\xrightarrow{uaw}0)$ in $F$ for every bounded $uaw$-null net(sequence) $(x_\alpha)((x_n))$ in $E$ \cite{O:19}. Let us continue with several ideal properties.
\begin{proposition}\label{}
Let $S : E \rightarrow F$ and $T : F \rightarrow G$ be two operators between Banach lattices.

$(1)$ If $S$ is $AM$-($\sigma$)-$un$-compact or $o$-($\sigma$)-$uaw$-compact and $T$ is compact, then $T\circ S$ is $AM$-compact, moreover it is order weakly compact.

$(2)$ If $S$ is $AM$-($\sigma$)-$un$-compact or $o$-($\sigma$)-$uaw$-compact, $T$ is weakly compact, then $T\circ S$ is order weakly compact.

$(3)$ If $S$ is $AM$-($\sigma$)-$un$-compact or $o$-($\sigma$)-$uaw$-compact, $T$ is $AM$-($\sigma$)-$un$-compact, then $T\circ S$ is $AM$-($\sigma$)-$un$-compact, moreover it is $o$-($\sigma$)-$uaw$-compact.

$(4)$ If $S$ is $AM$-($\sigma$)-$un$-compact or $o$-($\sigma$)-$uaw$-compact, $T$ is $o$-($\sigma$)-$uaw$-compact, then $T\circ S$ is $o$-($\sigma$)-$uaw$-compact.

$(5)$ If $S$ is $AM$-$un$-compact and $T$ is strong (unbounded norm) continuous operator, then $T\circ S$ is $AM$-($un$)-compact.

$(6)$ If $S$ is $o$-$uaw$-compact and $T$ is unbounded Dunford-Pettis (unbounded absolute weak continuous) operator, then $T\circ S$ is $AM$-$un$ ($o$-$uaw$)-compact.	
\end{proposition}
It is nature to ask that whether every adjoint operator of $AM$-$un$ ($o$-$uaw$)-compact operator is $AM$-$un$ ($o$-$uaw$)-compact. The answer is no. Considering identical operator $I:l_1\rightarrow l_1$ is $AM$-$un$-compact and $o$-$uaw$-compact, but $I^{\prime}:l_\infty\rightarrow l_\infty$ is not $AM$-$un$-compact and $o$-$uaw$-compact.

A Banach space is said to have the \emph{Schur property} whenever $x_n\xrightarrow{w}0$ implies $x_n\rightarrow0$.

\begin{proposition}\label{}
Let $F$ be Banach lattice with order continuous dual, then the adjoint operator of any continuous operator $T:c_0\rightarrow F$ is $AM$-$un$-compact, moreover it is $o$-$uaw$-compact.	
\end{proposition}
\begin{proof}
For any order interval of $F^{\prime}$, it is weakly compact subset of $F^{\prime}$, $T^{\prime}$ maps it to a relatively weakly compact subset of $l_1$. Since $l_1$ has Schur property, hence $T^{\prime}$ is $AM$-$un$-compact.	
\end{proof}
A element $e\in E_+$ is called an \emph{atom} of the Riesz space $E$ if the principal ideal $E_e$ is one-dimensional. $E$ is called an atomic Banach lattice if it is the band generated by its atoms. According to every order interval in atomic order continuous Banach lattice is compact, then we have the following result.
\begin{proposition}\label{}
Let $E$ be Banach lattice, $F$ be Banach lattice with atomic order continuous dual, then the adjoint operator of any continuous operator $T:E\rightarrow F$ is $AM$-$un$-compact, moreover it is $o$-$uaw$-compact.	
\end{proposition}
It is clear that every adjoint operator of compact operator is compact, it follows the next result.
\begin{proposition}\label{}
If $E$ and $F$ are Banach lattices with order unit, then the adjoint of any $AM$-$un$-compact operator from $E$ into $F$ is $AM$-$un$-compact.
\end{proposition}	

In \cite{WC:19}, a net $(x_\alpha^\prime)$ in dual Banach lattice $E^\prime$ is said to be \emph{unbounded absolute weak*} convergent ($uaw^*$-convergent, for short) to some $x^\prime\in E^\prime$ whenever $|x^\prime_\alpha-x^\prime|\wedge u^\prime\xrightarrow{w^*}0$ for any $u^\prime\in E^\prime$. We can also form $uaw^*$-topology by 
$$V_{\epsilon,u^\prime,x}=\{x^\prime\in E^\prime:(|x^\prime|\wedge u)(x)<\epsilon\}$$
where $u^\prime\in E^\prime_+$, $\epsilon>0$, $x\in E_+$, which form a base of zero neighborhoods for a Hausdorff topology, and the convergence in this topology is exactly the $uaw^*$-topology. A bounded subset $A$ in Banach lattice $E^\prime$ is said (sequentially) $uaw^*$-compact whenever every net(sequence) $\{x^\prime_\alpha\}(\{x^\prime_n\})$ in $E^\prime$ has a subnet (subsequence), which is $uaw^*$-convergent. Then we using the $uaw$-topology to describe the adjoint of $AM$-$un$-compact operator.

\begin{proposition}\label{}
Let $T:E\rightarrow F$ be a order bounded operator between two Banach lattice, then the following hold:

$(1)$ If $T$ is $AM$-$un$-compact operator, then $T^\prime$ maps order interval to a relatively $uaw^*$-compact subset of $E^\prime$. 

$(2)$ If $T^\prime$ is $AM$-$un$-compact operator, then $T$ maps order interval to relatively $uaw$-compact subset of $F$. 	
\end{proposition}
\begin{proof}
We prove the first part, another is similar. Since $T$ is order bounded, then $T$ is $AM$-compact, by \cite[Theorem~3.27]{AB:06},  for any net of $T^\prime[-y,y]$ for some $y\in F_+$, there exsit a subnet is uniform convergence on every order interval of $E$. According to \cite[Theorem~3.55]{AB:06} and $T^\prime$ is order bounded, then $T^\prime (B_{F^\prime})$ is relatively $uaw^*$-compact. 	
\end{proof}

\section{operator space of $AM$-$un$-compact operators}\label{}
\begin{proposition}\label{}
Let $L(E,F)$ denote the collection of continuous operators from Banach latticce $E$ into Banach lattice $F$, the set of all $AM$-($\sigma$)-$un$-compact in $L(E,F)$ is a closed subspace.
\end{proposition}
\begin{proof}
Let $(T_n)$ be a sequence of $AM$-$un$-compact operators in $L(E,F)$ and $T_n\rightarrow T$. For any order bounded net $(x_\alpha)$, for every $n$, the net $(T_mx_\alpha)_\alpha$ has a $un$-convergent subnet. By a standard diagonal argument, we can find a common subnet for all these nets. Passing to this subnet, we may assume without loss of generality that for every $m$ we have $T_mx_\alpha\xrightarrow{un} y_m$ for some $y_m$. Note that
$$\Vert y_m-y_n\Vert\leq \lim\inf_\alpha\Vert T_mx_\alpha-T_nx_\alpha\Vert\leq\Vert T_m-T_n\Vert\rightarrow0,$$
so that the sequence $(y_m)$ is Cauchy and, therefore, $y_m\rightarrow y$ for some $y\in F$. Fix $u\in F_+$ and $\epsilon >0$. Find $m_0$ such that $\Vert T_{m0}-T\Vert<\epsilon$ and $\Vert y_{m0}-y\Vert<\epsilon$. Find $\alpha_0$ such that $\big\Vert|T_{m0}x_\alpha-y_{m0}|\wedge u\big\Vert<\epsilon$ whenever $\alpha\geq\alpha_0$. It follows from 
$$|Tx_\alpha-y|\wedge u\leq |Tx_\alpha-T_{m0}x_\alpha|+|T_{m0}x_\alpha-y_{m0}|\wedge u+|y_{m0}-y|$$
that $\big\Vert|Tx_\alpha-y|\wedge u\big\Vert<3\epsilon$, hence $Tx_\alpha\xrightarrow{un}y$.
\end{proof}
\begin{proposition}\label{}
Let $E$ be Riesz space and $F$ be Banach lattice, for two positive operators $T$ and $S$ satisfying $0\leq S\leq T$ and $T$ is lattice homomorphism, if $T$ is $AM$-($\sigma$)-$un$-compact, then $S$ is also. 	
\end{proposition}
\begin{proof}
We prove the sequential case, the net case is similar. Let $(x_n)$ be order bounded sequence in $E$, passing to a subsequence, we may assume that $(Tx_n)$ is $un$-convergent in $F$, so it is $un$-Cauchy. According that
$$|Sx_n-Sx_m|\wedge u\leq S|x_n-x_m|\wedge u\leq T|x_n-x_m|\wedge u=|Tx_n-Tx_m|\wedge u\rightarrow0$$
as $n,m\rightarrow\infty$. It follows that $(Sx_n)$ is $un$-Cauchy, since the convergence and $un$-convergence are equivalent on order bounded sequence and $F$ is norm complete, therefore $S$ is $AM$-$\sigma$-$un$-compact.	
\end{proof}
Using the same method, and by \cite[Theorem~7.2]{W:99}, the following result are also true.
\begin{corollary}\label{}
Let $E$ be Riesz space and $F$ be $KB$-space, for two positive operators $T$ and $S$ satisfying $0\leq S\leq T$ and $T$ is lattice homomorphism, if $T$ is $AM$-($\sigma$)-$uaw$-compact, then $S$ is also.	
\end{corollary}
We now turn our discussion to lattice properties of $AM$-$un$-compact operators. We may ask whether or not a $AM$-$un$-compact operators between two Banach lattices possesses a modulus. Obviously, compact operator is $AM$-$un$-compact. According to the example of U.Krengel in \cite{K:66}, a $AM$-$un$-compact operator need not have a modulus and also that a $AM$-$un$-compact operator may have a modulus which is not $AM$-$un$-compact.

So we consider the order bounded $AM$-$un(uaw)$-compact operaotrs.
We mentioned above that order bounded $AM$-$un$-compact operator is $AM$-compact.
By \cite[Proposition~3.7.2]{MN:91}, we have the following result.
\begin{proposition}\label{}
Let $E$ be Banach lattice, $F$ be order continuous Banach lattice, then the collection of all order bounded $AM$-($\sigma$)-$un$-compact operators from $E$ into $F$ forms a band in $L^b(E,F)$.	
\end{proposition}
When $F$ is order continuous, $o$-$uaw$-compact operator is $AM$-$un$-compact, hence we have the following corollary.
\begin{corollary}\label{}
Let $E$ be Banach lattice, $F$ be order continuous Banach lattice, then the collection of all order bounded $o$-($\sigma$)-$uaw$-compact operators from $E$ into $F$ forms a band in $L^b(E,F)$.	
\end{corollary}

Next, we study the topology of $AM$-$un$-compact operators and $o$-$uaw$-compact operators.

Recall that a operator $T:E\rightarrow F$ between two Riesz space $E$ and $F$ is positive if $T(x)\geq 0$ in $F$ whenever $x\geq 0$ in
$E$. An operator $T : E\rightarrow F$ is regular if $T = T_1-T_2$ where $T_1$ and $T_2$ are positive
operators from $E$ into $F$. 
As usual, $L^r(E,F)$ denotes the space of all regular operators from
$E$ into $F$. With respect to the uniform operator norm the space $L^r$ is not Banach in
general, but there exists a natural norm on $L^r$, is defined by
$$\Vert T\Vert_r:=\big\Vert|T|\big\Vert:=\sup\{\big\Vert|T|x\big\Vert:\Vert x\Vert\leq1\}.$$
The regular norm (r-norm, for short) $\Vert\cdot\Vert_r$, which turns
$L^r(E,F)$ into a Banach lattice.
If $F$ is Dedekind complete, it is clear that r-norm is a lattice norm on $L^b(E,F)=L^r(E,F)$ and $L^b(E,F)$ is a Banach lattice under the r-norm \cite[Theorem~4.74]{AB:06}.

The collection of all $AM$ ($\sigma$)-unbounded norm compact operators of $L(E, F)$ will be denoted by $K_{AM-(\sigma)-un}(E, F)$. That is,
$$K_{AM-(\sigma)-un}(E, F)=\{T \in L(E, F) : T \text { is AM-($\sigma$)-un compact}\}$$
respectively,
$$K^r_{AM-(\sigma)-un}(E, F)=\{T\in K_{AM-(\sigma)-un}(E, F):\pm T\leq S\in K^+_{AM-(\sigma)-un}(E, F)\}$$

A Banach lattice $E$ is said to be $AL$-space if $\Vert x+y\Vert=\Vert x\Vert+\Vert y\Vert$ holds for all $x,y\in E_+$. A Banach lattice $E$ is said to be $AM$-space if $\Vert x+y\Vert=\max\{\Vert x\Vert,\Vert y\Vert\}$ holds for all $x,y\in E_+$. 

Now, we research that when $AM$-$un$-compact operators form a Banach lattice. 
\begin{proposition}\label{}
Let $T:E\rightarrow F$ is a $AM$-$un$-compact operator from Banach lattice $E$ into Banach lattice $F$ which has strong order unit, then the modulus of $T$ is exsits and is a $AM$-$un$-compact operator.	
\end{proposition}
\begin{proof}
If $F$ has strong order unit, then $F$ lattice and norm isomorphic to $C(K)$ for some compact Hausdorff space $K$ by \cite[Theorem~4.21 and Theorem~4.29]{AB:06}, we identify $C(K)$ with $F$. Since 
$$|T|(x)=\sup T[-x,x]$$
for $x\in E_+$ and $T$ is $AM$-$un$-compact, by \cite[Theorem~2.3]{KMT:16}, $T[-x,x]$ is relatively compact subset of $F$. According to \cite[Theorem~4.30]{AB:06}, the modulus of $T$ exists. And by the proof of the second part of \cite[Theorem~5.7]{AB:06}, $|T|$ is compact, hence it is $AM$-$un$-compact.	
\end{proof}

The space $K^r(E, F)$ is not a Banach space under the regular norm. But there is a so-called k-norm on $K^r(E, F):\Vert T\Vert_k = inf\{\Vert S\Vert :	S \in K^+(E, F), S \geq \pm T\}$ under which it is Banach space  (see \cite{CW:97,W:95} for details). It is natural to consider the corresponding problem for $AM$-$un$-compact operators.

We introdunce a new norm on $AM$-$un$-compact operators.
\begin{definition}\label{}
If $E$ and $F$ are Banach lattices then the $AM$-$un$-norm on $K_{AM-un}(E,F)$ is
defined by	
$$\Vert T\Vert_{AM-un}=\inf\{\Vert S\Vert:S\in K_{AM-un}^+(E,F),S\geq \pm T\}.$$
It is clear that this does define a lattice norm and that $\Vert T\Vert_{AM-un}\geq \Vert T\Vert_r$.
\end{definition}

\begin{theorem}\label{}
For Banach lattices $E$ and $F$, $K^r_{AM-un}(E, F)$ is complete under the $AM$-$un$-norm.	
\end{theorem}
\begin{proof}
Let $(T_n)$ be Cauchy-sequence in $K^r_{AM-un}(E,F)$. For any $n\in N$, we have
$\Vert T_n-T\Vert_{AM-un}<2^{-n}.$
Since $\Vert T\Vert\leq\Vert T\Vert_{AM-un}$, so $(T_n)$ is Cauchy-sequence under operator norm. From the definition of $AM$-$un$-norm, there exsits $(V_n)\subset K^r_{AM-un}(E,F)$ statisfying $\pm(T_n-T_{n+1})\leq V_n$, $\Vert V_n\Vert\leq2^{-n}$ hold for all $n\in N$. It is clear that $Q_n=\sum_{k=n}^{\infty}V_k$ is positive $AM$-$un$-compact operator and $\Vert Q_n\Vert\leq2^{-n+1}$. For any $x\in E_+$, we have $|(T_n-T)x|=\lim\limits_{m\rightarrow\infty}|(T_n-T_m)x|\leq \lim\limits_{m\rightarrow\infty}\sum_{k=n}^{m-1}|(T_k-T_{k+1})x|\leq Q_nx$, it follows $\pm(T_n-T)\leq Q_n$, hence $T\in K^r_{AM-un}(E,F)$. And Since $\Vert T_n-T\Vert_{AM-un}\leq \Vert Q_n\Vert<2^{-n+1}$, therefore $T$ is the limit of $(T_n)$ and $K^r_{AM-un}(E, F)$ is complete under the $AM$-$un$-norm.
\end{proof}

We do not know whether the $AM$-$un$-norm is equivalent to operator norm and r-norm.

Suppose that $T_n,T\in K^r_{AM-un}(E,F)$ and $\Vert T_n-T\Vert_{AM-un}\rightarrow0$, $|T_n|$ is exsits for every $n\in N$, whether the modulus of $T$ exsits?
\begin{theorem}\label{}
For Banach lattice $E$ and $F$ be Banach lattices, $T_n,T\in K^r_{AM-un}(E,F)$ and $\Vert T_n-T\Vert_{AM-un}\rightarrow0$, $|T_n|$ is exsits for every $n\in N$, then the modulus of $T$ exsits, moreover
$$\big\Vert |T_n|-|T|\big\Vert_{AM-un}\rightarrow0$$	
\end{theorem}
\begin{proof}
Since $\Vert T_n-T\Vert_{AM-un}\rightarrow0$, so ${T_n}$ is $\Vert\cdot\Vert_{AM-un}$-Cauchy sequence and there exsits $H_{nm}\in K^r_{AM-un}(E,F)$ satisfying that $\pm(T_n-T_m)\leq H_{nm}$ and $\Vert H_{nm}\Vert\rightarrow0$. And since
$$\pm T_n=\pm(T_n-T_m)\mp T_m\leq H_{nm}+|T_m|,$$
so $|T_n|\leq H_{nm}+|T_m|$, similarly, $|T_m|\leq H_{nm}+|T_n|$, hence $\pm(|T_n|-|T_m|)\leq H_{nm}$ and
$$\big\Vert|T_n|-|T_m|\big\Vert_{AM-un}\leq \Vert H_{nm}\Vert\rightarrow0$$, i.e. ${|T_n|}$ is $\Vert\cdot\Vert_{AM-un}$-Cauchy sequence.

Let $|T_n|\rightarrow S$, we have to show that $S=|T|$. In fact, if $V_n\in K^+_{AM-un}(E,F)$ satisfying $\pm(|T_n|-S)\leq V_n$ and $\Vert V_n\Vert\rightarrow0$, then we have $\pm T_n\leq |T_n|\leq S+V_n$, so $\pm T\leq S$.

In orther hand, if $U\in K^+_{AM-un}(E,F)$, $U\geq \pm T$, then there exsits $H_n\in K^+_{AM-un}(E,F)$ such that $\pm(T_n-T)\leq H_n$ and $\Vert H_n\Vert\rightarrow0$. According to $\pm T_n=\pm(T_n-T)\mp T\leq H_n+U$, we have $|T_n|\leq H_n+U$, hence $S\leq U$, moreover $S=|T|$, therefore $\big\Vert |T_n|-|T|\big\Vert_{AM-un}\rightarrow0$.
\end{proof}

\begin{corollary}\label{}
Let $E$ be an arbitrary Banach lattice, $F$ be $AM$-space or order continuous, then $K^r_{AM-un}(E,F)$ is a Banach lattice under the $AM$-$un$-norm.	
\end{corollary}
If $f\in E^\thicksim$ and $u\in F$, then the symbol $f\otimes u$ will denote the order bounded operator defined by $(f\otimes u)(x):=f(x)u$. For a compact Hausdorff space $K$, $K$ is called \emph{extremely disconnected} or \emph{Stonian} if the closure of every open set is open. Then, we study that when $AM$-$un$-compact operators form $AM$-space.
\begin{theorem}\label{}
Let $E$ and $F$ be Banach lattices, then $K^r_{AM-un}(E,F)$ is $AM$-space under the $AM$-$un$-norm if and only if $E$ is $AL$-space and $F$ is $AM$-space.	
\end{theorem}
\begin{proof}
$\Leftarrow$ Assume that $E$ is $AL$-space and $F$ is $AM$-space. If $T,S\in K^+_{AM-un}(E,F)$ and $x\in E_+$, then 
$$(S\vee T)(x)=\sup\{Sx_1+Tx_2:0\leq x_1,x_2,x_1+x_2=x\}.$$
Since $E$ is $AL$-space, so $\Vert Sx_1+Tx_2\Vert\leq\Vert Sx_1\Vert+\Vert Tx_2\Vert\leq \Vert S\Vert\Vert x_1\Vert+\Vert T\Vert\Vert x_2\Vert\leq (\Vert S\Vert\vee\Vert T\Vert)(\Vert x_1\Vert+\Vert x_2\Vert)=(\Vert S\Vert\vee\Vert T\Vert)(\Vert x\Vert)$.

Suppose that $x_1^1,x_2^1...x_1^n,x_2^n\in [0,x]$ satisfying $x_1^k+x_2^k=x$ for $1\leq k\leq n$. Since $F$ is $AM$-space, so we have
$$\Vert\vee^n_{k=1}(Sx^k_1+Tx_2^k)\Vert= \vee^n_{k=1}\Vert Sx^k_1+Tx_2^k\Vert\leq \max\{\Vert S\Vert,\Vert T\Vert\}\Vert x\Vert,$$
hence $\Vert(S\vee T)(x)\Vert\leq \max\{\Vert S\Vert,\Vert T\Vert\}\Vert x\Vert$, therefore $\Vert(S\vee T)\Vert\leq \max\{\Vert S\Vert,\Vert T\Vert\}$, by $\Vert(S\vee T)\Vert\geq \max\{\Vert S\Vert,\Vert T\Vert\}$, we have $\Vert(S\vee T)\Vert=\max\{\Vert S\Vert,\Vert T\Vert\}$, therefore $K^r_{AM-un}(E,F)$ is $AM$-space.

$\Rightarrow$ We show that $E^{\prime}$ is $AM$-space. For some $f\in F_+$ satisfying $\Vert f\Vert=1$, let $x_1^{\prime},x_2^\prime\in E^\prime$, since $K^r_{AM-un}(E,F)$ is $AM$-space, hence
$$\Vert x_1^\prime\vee x_2^\prime\Vert=\Vert (x_1^\prime\vee x_2^\prime)\otimes f\Vert=\Vert x_1^\prime\otimes f\Vert\vee\Vert x_2^\prime\otimes f\Vert=\Vert x_1^\prime\Vert\vee\Vert x_2^\prime\Vert.$$
So $E^\prime$ is $AM$-space, moreover $E$ is $AL$-space by \cite[Theorem~4.23]{AB:06}.

Finally, we show that $F$ is $AM$-space. For some $x^\prime\in E^\prime_+$ such that $\Vert x^\prime\Vert=1$, let $f_1,f_2\in F_+$, since $K^r_{AM-un}(E,F)$ is $AM$-space, hence
$$\Vert f_1\vee f_2\Vert=\Vert x^\prime\otimes(f_1\vee f_2)\Vert=\Vert x^\prime\otimes f_1\Vert\vee \Vert x^\prime\otimes f_2\Vert=\Vert f_1\Vert\vee\Vert f_2\Vert.$$
So $F$ is $AM$-space.
\end{proof}
Using the same method, we have the dual theorem.
\begin{theorem}\label{}
Suppose that $E$ and $F$ be Banach lattices, then $K^r_{AM-un}(E,F)$ is $AL$-space under the $AM$-$un$-norm if and only if $E$ is $AM$-space and $F$ is $AL$-space.	
\end{theorem}
\begin{proof}
$\Leftarrow$ Assume that $E$ is $AM$-space and $F$ is $AL$-space. If $T,S\in K^+_{AM-un}(E,F)$ and for any $\epsilon>0$, there exsits $x_1,x_2\in E_+$ such that $\Vert x_1\Vert,\Vert x_2\Vert\leq1,\Vert Sx_1\Vert\geq\Vert S\Vert-\epsilon$ and $\Vert Tx_2\Vert\geq\Vert T\Vert-\epsilon$. Since $E$ is $AM$-space, so $\Vert x_1\vee x_2\Vert=\max\{\Vert x_1\Vert,\Vert x_2\Vert\}\leq1$. And since $F$ is $AL$-space, so 
$$\Vert(S+T)(x_1\vee x_2)\Vert=\Vert S(x_1\vee x_2)\Vert+\Vert T(x_1\vee x_2)\Vert\geq \Vert S(x_1)\Vert+\Vert T(x_2)\Vert\geq \Vert S\Vert+\Vert T\Vert-2\epsilon,$$
so $\Vert S+T\Vert\geq \Vert S\Vert+\Vert T\Vert$, and by $\Vert S+T\Vert\leq \Vert S\Vert+\Vert T\Vert$, we have $\Vert S+T\Vert=\Vert S\Vert+\Vert T\Vert$, hence $K^r_{AM-un}(E,F)$ is $AL$-space.

$\Rightarrow$ We show that $F$ is $AL$-space. For some $x^\prime \in E^\prime_+$ and $\Vert x^\prime\Vert=1$, and any $f_1,f_2\in F_+$, since $K^r_{AM-un}(E,F)$ is $AL$-space, we have 
$$\Vert f_1+f_2\Vert=\Vert x^\prime\otimes(f_1+f_2)\Vert=\Vert x^\prime\otimes f_1\Vert+\Vert x^\prime\otimes f_2\Vert=\Vert f_1\Vert+\Vert f_2\Vert,$$
hence $F$ is $AL$-space.

In final, we show that $E$ is $AM$-space. For some $f\in F_+$ and $\Vert f\Vert=1$, and any $x_1^\prime,x_2^\prime\in F_+$, since $K^r_{AM-un}(E,F)$ is $AL$-space, we have 
$$\Vert x_1^\prime+x_2^\prime\Vert=\Vert (x_1^\prime+x_2^\prime)\otimes f\Vert=\Vert x_1^\prime\otimes f\Vert+\Vert x_2^\prime\otimes f\Vert=\Vert x_1^\prime\Vert+\Vert x_1^\prime\Vert,$$
hence $E^\prime$ is $AL$-space, so $E$ is $AM$-space by \cite[Theorem~4.23]{AB:06}.
\end{proof}

\end{document}